\documentclass[12pt, reqno]{amsart}
\usepackage[a4paper,vmargin={3cm,3cm},hmargin={2.5cm,2.5cm},bottom=30mm]{geometry}
\usepackage{amsfonts}
\usepackage{amsthm,amsmath}
\usepackage{amsbsy}
\usepackage{bbm}
\usepackage{natbib}
\usepackage{amssymb}
\usepackage{verbatim}
 \usepackage{color}
\usepackage{csquotes}
\usepackage{mathabx}

\usepackage[mathscr]{eucal}
\usepackage{times}
\usepackage[pdfpagemode=UseNone,bookmarksopen=false,colorlinks=blue]{hyperref}
\hypersetup{ colorlinks=true, linkcolor=blue, citecolor=blue,
  filecolor=blue, urlcolor=blue}
  
\newtheorem{theorem}{Theorem}[section]

\theoremstyle{definition}
\newtheorem{definition}[theorem]{Definition}
\theoremstyle{corollary}
\newtheorem{corollary}[theorem]{Corollary}
\theoremstyle{example}

\theoremstyle{note}

\theoremstyle{notation}

\numberwithin{equation}{section}

\bibfont{\footnotesize}

\makeatletter
\def\blfootnote{\gdef\@thefnmark{}\@footnotetext}
\makeatother

\begin{document}
\title[Image partition regular matrices, Solution in C-sets]
{Infinite image partition regular matrices - Solution in C-sets}
\author{Sukrit Chakraborty$^1$ and Sourav Kanti Patra$^2$}
\address{Indian Statistical Institute, 203, B.T. Road, Kolkata- 700108, India$^1$} 
\address{Ramakrishna Mission Vidyamandira, Belur Math, Howrah-711202, West Bengal, India$^2$}
\email{sukritpapai@gmail.com$^1$, souravkantipatra@gmail.com$^2$}
\blfootnote{\textup{2010} \textit{Mathematics Subject Classification}: Primary : 05D10 Secondary : 22A15}

\newcommand{\acr}{\newline\indent}
\begin{abstract}
A finite or infinite matrix $A$ is image partition regular provided that whenever $\mathbb{N}$ is finitely colored, there must be some $\overset{\rightarrow}{x}$ with entries from $\mathbb{N}$ such that all entries of $A \overset{\rightarrow}{x}$ are in the same color class. Comparing to the finite case, infinite image partition regular matrices seem more harder to analyze. The concept of centrally image partition regular matrices were introduced to extend the results of finite image partition regular matrices to infinite one. In this paper, we shall introduce the notion of C-image partition regular matrices, an interesting subclass of centrally image partition regular matrices. Also we shall see that many of known centrally image partition regular matrices are C-image partition regular.   
\end{abstract}
\maketitle

\section{Introduction}
The classical theorems of Ramsey Theory can be spontaneously stated as statements about image partition regular matrices which is why the study of Image partition regular matrices seeks remarkable attention. For example, Schur’s Theorem [\cite{schur1917kongruenz}] and the length $4$ version of van der Waerden’s Theorem [\cite{van1927beweis}] assures us that  the matrices 
$$\begin{pmatrix}
1 & 0 \\ 0 & 1 \\ 1 & 1 
\end{pmatrix} \text{ and }
\begin{pmatrix}
1 & 0 \\ 1 & 1\\ 1 & 2 \\ 1 & 3
\end{pmatrix}
$$
are image partition regular.

Let us recall the following well known definition of the image partition regularity.
\begin{definition}
Let $A$ be a $m \times n$ matrix with entries from $\mathbb{Q}$ for some $m,n \in \mathbb{N}$. We call the matrix $A$ to be an image partition regular matrix over $\mathbb{N}$ if and only if there exists $i \in \{1,2, \cdots , p\}$ and $\overset{\rightarrow}{x} \in \mathbb{N}^{n}$ such that $A \overset{\rightarrow}{x} \in E_i^m$, whenever we are able to write $\mathbb{N} = \bigcup_{i=1}^{p}E_i$ for some $p \in \mathbb{N}$.
\end{definition}
Several characterizations of finite image partition regular matrices involve the notion of \enquote{first entries matrix}, a concept based on Deuber’s $(m, p, c)$ sets.

We recall the following definition from \cite{iprm2000hindman}.
\begin{definition}
Let $A$ be a $p \times q$ matrix with rational entries. Then $A$ is a first entries matrix if and only if no row of $A$ is $\overset{\rightarrow}{0}$ and there exist strictly positive rational numbers $t_1, t_2, \cdots , t_q $ such that $a_{i,j} = t_j$, where $i \in \{1, 2, \cdots , p\}$ and $j = \min\{l \in \{1, 2, \cdots , q\} : a_{i,l} \neq 0\}$.

We call $t_j$ to be the first entry of $A$ if there exists $i \in \{1, 2, \cdots , p\}$ such that $j = \min\{l \in \{1, 2, \cdots , q\} : a_{i,l} \neq 0 \}$.
\end{definition}
Some of the known characterizations of finite image partition regular matrices involve the notion of central sets while the  famous concept was introduced in [\cite{frus81}] and defined in terms of the view point of topological dynamics. These sets enjoy very strong combinatorial properties. (See [\cite{frus81}, Proposition $8.21$] or [\cite{hindman1998algebra}, Chapter $14$].) They have a nice
characterization in terms of the algebraic structure of $\beta \mathbb{N}$, the Stone-{\v{C}}ech compactification
of $\mathbb{N}$. We shall present this characterization below, after introducing the necessary background information.

Let $(S, +)$ be an infinite discrete semigroup. The points of $\beta S$ are taken to be the ultrafilters on $S$ with the understanding that the principal ultrafilters are being identified with the points of $S$. It is a folklore that for a given set $A \subseteq S$, $\overline{A} = \{p \in \beta S : A \in p\}$. The set $\{A : A \subseteq S\}$ turns out to be a basis for the open sets (as well as a basis for the closed sets) of $\beta S$.

One can naturally extend of the operation $+$ of $S$ to the whole of $\beta S$ by making $\beta S$ a compact right topological semigroup with its topological center containing $S$. This says that for each $p \in \beta S$ the function $\mu_p : \beta S \to \beta S$ is continuous and for each $x \in S$, the function $\nu_x : \beta S \to \beta S$ is continuous, where $\mu_p(q) = q + p$ and $\nu_x(q) = x + q$. Given $p,q \in \beta S$ and $A \subseteq S$, $A \in p+q$ if and only if $\{x \in S : -x +A \in q\} \in p$, where $-x +A = \{y \in S : x+y \in A\}$.

If a non-empty subset $I$ of a semigroup $(T,+)$ satisfies $T + I \subseteq I$, we call $I$ to be a left ideal of $T$ and a right ideal if $I + T \subseteq I$. A two sided ideal (or simply an ideal) is both a left and a right ideal. A minimal left ideal is a left ideal that does not contain any proper left ideal. Now it is very obvious to analogously define a minimal right ideal and the smallest ideal.

Any compact Hausdorff right topological semigroup $(T,+)$ contains idempotents and therefore has a smallest two sided ideal 
\begin{align*}
K(T) &= \bigcup\{\mathcal{L} : \mathcal{L}\text{ is a minimal left ideal of T}\}\\
&= \bigcup\{\mathcal{R} : \mathcal{R}\text{ is a minimal right ideal of T}\}. 
\end{align*}  
Given a minimal left ideal $\mathcal{L}$ and a minimal right ideal $\mathcal{R}$, it turns out that $\mathcal{L} \cap \mathcal{R}$ is a group and therefore contains an idempotent. If $p$ and $q$ are idempotents in $T$, we write $p \leqslant q$ if and only if $p + q = q + p =p$. An idempotent is minimal with respect to this relation if and only if it is a member of the smallest ideal $K(T)$ of $T$.

A beautiful inauguration to the algebra of $\beta S$ is given in \cite{hindman1998algebra}.
\begin{definition}
Let $(S,+)$ be an infinite discrete semigroup. A subset in $S$ is said to be Central if and only if it is contained in some minimal idempotent of $(\beta S, +)$.
\end{definition}
Now we state the most general version of Central Sets Theorem from \cite{de2008new}. We state it here only for the commutative subgroups.
\begin{theorem} \label{thm1.4}
Let $(S,+)$ be a commutative semigroup and denote the set of all sequences in $S$ by $\tau$. Let $C \subseteq S$ be central , then there exists functions $\alpha : \mathcal{P}_f(\tau)\to S$ and $H: \mathcal{P}_f(\tau) \to \mathcal{P}_f(\mathbb{N})$ such that 
\begin{enumerate}
\item \label{1.41}if $F,G \in \mathcal{P}_f(\tau)$ and $F \subsetneqq G$ then $\max H(F) < \min H(G)$, and 
\item \label{1.42}whenever $m \in \mathbb{N}$, $G_1, G_2, \cdots , G_m \in \mathcal{P}_f(\tau)$, $G_1 \subsetneqq G_2 \subsetneqq \cdots \subsetneqq G_m$ and for each $i \in \{1,2, \cdots , m\}$, $f_i \in G_i$, one has $$\sum_{i=1}^{m}\big(\alpha(G_i)+\sum_{t\in H(G_i)}f_i(t)\big)\in C.$$
\end{enumerate}
\end{theorem}
Recently a lot of attention has been paid to those sets which satisfy the conclusion of the latest Central Sets Theorem. Like Central sets they also contain the image of finite image partition regular matrices.  
\begin{definition} \label{def1.5}
Let $(S,+)$ be a commutative semigroup. Also let $\tau = \mathbb{N}_S$ be the set of all sequences in $S$ and  $C \subseteq S$. If there exists functions $\alpha : \mathcal{P}_f(\tau)\to S$ and $H: \mathcal{P}_f(\tau) \to \mathcal{P}_f(\mathbb{N})$ such that the conditions \ref{1.41} and \ref{1.42} of Theorem \ref{thm1.4} is satisfied, then we call $C$ to be a $C$ set. 
\end{definition}
Therefore we can readily observe that Central sets are in particular C-sets. We now present some notations from \cite{de2008new}.
\begin{definition}
Let $(S,+)$ be a commutative semigroup, and let $\tau$ be as in the Definition \ref{def1.5}.
\begin{enumerate}
\item $A \subseteq S$ is said to be a J-set if for every $F \in \mathcal{P}_f(\tau)$ there exists $a \in S$ and $H \in \mathcal{P}_f(\mathbb{N})$ such that for all $f \in F$, $$a + \sum_{t \in H}f(t) \in A.$$
\item $J(S) = \{p \in \beta S : (\forall A \in p)(A \text{ is a J-set})\}$.
\end{enumerate}
\end{definition}
\begin{theorem}
Let $(S,+)$ be a discrete commutative semigroup and $A$ be a subset of $S$. Then $A$ is a J set if and only if $J(S) \bigcap ClA \neq \phi$.  
\end{theorem}
\begin{proof}
The theorem follows from Theorem $3.11$ of \cite{hindman1998algebra} while noting that the collection of J sets form a partition regular family.
\end{proof}
The following is a consequence of Theorem $3.8$ of \cite{de2008new}.
\begin{theorem}\label{thm1.8}
Let $(S,+)$ be a commutative semigroup and $A \subseteq S$. Then $A$ is a C-set if and only if $clA \bigcap J(S)$ contains atleast an idempotent $p$. 
\end{theorem}
We also state the following theorem which is Theorem $3.5$ in \cite{de2008new}.
\begin{theorem}
Let $(S,+)$ be a discrete commutative semigroup, then $J(S)$ is a closed two sided ideal of $\beta S$ and $Cl K(\beta S) \subseteq J(S)$.
\end{theorem}
It is known that finite image partition regular matrices are compatible with respect to central sets. A similar result holds true for C-sets.
\begin{theorem} \label{thm1.10}
Let $A$ be a $p \times q$ matrix with rational entries for some natural numbers $p$ and $q$. The following statements are equivalent:
\begin{enumerate}
\item $A$ is an image partition regular matrix.
\item For any C-set $C$ of $\mathbb{N}$, $A\overset{\rightarrow}{x} \in C^p$ for some $\overset{\rightarrow}{x} \in \mathbb{N}^q$.
\item For any C-set $C$ of $\mathbb{N}$, $\{\overset{\rightarrow}{x} \in \mathbb{N}^q : A \overset{\rightarrow}{x} \in C^p\}$ is also a central subset $\mathbb{N}^v$.
\end{enumerate}
\end{theorem}
\begin{proof}
The proof is similar as that of Theorem $1.2$ of \cite{iprm2}.
\end{proof}
The notion of image partition regular matrices extends naturally to infinite $\omega \times \omega$ matrices provided each row of the matrix contains only finitely many non-zero entries. (Here $\omega$, the first infinite cardinal, is also the set of nonnegative integers.)

It is an immediate consequence of Theorem $1.5(b)$ of \cite{iprm2000hindman} that whenever $A$ and $B$ are finite image partition regular matrices, so is $\begin{pmatrix}
A & \textbf{O}\\\textbf{O} & B \end{pmatrix}$, where $\textbf{O}$ represents a matrix of the appropriate size with all zero entries. However, the analogous result does not hold true for infinite image partition regular matrices because of  Theorem $3.14$ of \cite{deu95inf}.

Motivated by this distinction and by the condition of Theorem $1.5(l)$ of \cite{iprm2000hindman}, Hindman, Leader and strauss came up with the notion of \enquote{Centrally image partition regular matrices} and \enquote{Strongly Centrally image partition regular matrices} in Definition $2.7$ and in  Definition $2.10$ respectively in \cite{iprm2}.

In Section \ref{sec:s}, we shall introduce the notion of C-image partition regularity and see that the behaviour of this infinite image partition regularity almost same like Centrally image partition regularity. In Section \ref{sec:s:s} we shall give some classes of C-image partition regular matrices which also had occurred in the case of Centrally image partition regularity.  
\section{C-image partition regularity} \label{sec:s}   
Centrally image partition regular matrices were introduced in order to extend the results of finite image partition regular matrices to infinite one. In this section we shall introduce the notion of C-image partition regularity and see that parallel results for this type of image partition regularity also holds true.
\begin{definition}
Let $A$ be an $\omega \times \omega$ matrix with entries from $\mathbb{Q}$.
\begin{enumerate}
\item The matrix $A$ is C-image partition regular if and only if for every C-set $C$ of $\mathbb{N}$, one has $A\overset{\rightarrow}{x} \in C^{\omega}$ for some $\overset{\rightarrow}{x} \in \mathbb{N}^{\omega}$.
\item The matrix $A$ is strongly C-image partition regular if and only if for every C-set $C$ of $\mathbb{N}$, one has $\overset{\rightarrow}{y} = A\overset{\rightarrow}{x} \in C^{\omega}$ for some $\overset{\rightarrow}{x} \in \mathbb{N}^{\omega}$ and entries of $A\overset{\rightarrow}{x}$ corresponding to distinct rows of $A$ are distinct i.e., for all $i,j$ if row $i$ and row $j$ of $A$ are unequal then $y_i \neq y_j$.
\end{enumerate}
\end{definition}
Like Centrally image partition regular matrices, there is a simple necessary condition for a matrix to be strongly C-image partition regular which is as follows:
\begin{theorem}
Let $A$ be a strongly C-image partition regular matrix having no repeated rows. Then, $$\{i : \text{ for all }j \geqslant k, a_{i,j}= 0\} \text{ is finite for all }k \in \mathbb{N}.$$
\end{theorem}
\begin{proof}
Suppose that $\{i : \text{ for all }j \geqslant k, a_{i,j}= 0\}$ is finite. Then by discarding the other rows we may presume that $A$ is an $\omega \times k$ matrix. Let $D = \{\overset{\rightarrow}{x} \in \mathbb{N}^k : \text{ all entries of } A \overset{\rightarrow}{x} \text{ are distinct}\}$. Enumerate $D$ as $\langle \overset{\rightarrow}{x}^{(n)} \rangle_{n=1}^{\infty}$. Inductively choose distinct $y_n$ and $z_n$ in $A \overset{\rightarrow}{x}^{(n)}$ with $\{y_n, z_n\}\bigcap (\{y_t : t \in \{1,2, \cdots, n-1\}\} \bigcup \{z_t : t \in \{1,2, \cdots, n-1\}\}) \neq \phi$ if $n > 1$. Let $C = \{y_n : n \in \mathbb{N}\}$. Then there exists no $\overset{\rightarrow}{x} \in D$ with $A \overset{\rightarrow}{x} \in C^{\omega}$ and no $\overset{\rightarrow}{x} \in D$ with $A \overset{\rightarrow}{x} \in (\mathbb{N} \setminus  C)^{\omega}$.
\end{proof} 
\begin{theorem} \label{thA}
Let $p$ be an idempotent in $R$ where $R$ is a right ideal of $(\beta \mathbb{N}, +)$. Then for each $C \in p$, there are $2^c$ minimal idempotents in $R \bigcap \overline{C}$.
\end{theorem}
\begin{proof}
Let $C \in p$ and $C^* := \{x \in C : -x +C \in p\}$. Then notice that by Lemma $4.14$ of \cite{hindman1998algebra}, $-x + C^* \in p$, for each $x \in C^*$. For each $m \in \mathbb{N}$, $$S_m = 2^m\mathbb{N} \bigcap C^{*} \cap \bigcap \{-k +C^{*} : k \in C^{*} \cap \{1,2, \cdots, m\}\}.$$ Let $V = \bigcap_{m\in \mathbb{N}}\overline{S}_m$. For every $m \in \mathbb{N}$, $2^m \mathbb{N} \in p$ by Lemma $6.6$ of \cite{hindman1998algebra} and so $m \in p$. Thus $p \in V$. We show that $V$ is a subsemigroup of $\beta \mathbb{N}$, using Theorem $4.20$ of \cite{hindman1998algebra}. So, let $m \in \mathbb{N}$ and let $n \in S_m$. It is sufficient to show that $n + S_{m+n} \subseteq S_m$. Let $r \in S_{m+n}$. Obviously $n+r \in 2^m\mathbb{N}$. We have $n+r \in C^*$ because $n \in C^* \bigcap \{1,2, \cdots , m+n\}$. Let $k \in C^* \bigcap \{1,2, \cdots ,m\}$. Then $n \in -k+C^*$. So $k +n \in C^* \bigcap \{1,2, \cdots , m+n\}$ and thus $r \in -(k+n)+C^*$ so that $n+r \in -k + C^*$ as required. Since $p \in V$ we have by Theorem $6.32$ of \cite{hindman1998algebra} that $V$ contains a copy of $\mathbb{H} = \bigcap_{n=1}^{\infty}\overline{\mathbb{N}2^n}$. By Theorem $6.9$ of  \cite{hindman1998algebra}, $(\beta \mathbb{N}, +)$ has $2^c$ minimal left ideals. Thus there is a subset $W$ of $\beta \mathbb{N}$ containing idempotents such that $|W| = 2 ^c$. The subset $W$ will also have the property that whenever $u$ and $v$ are distinct members of $W$, $u + v =u$ and $v+u = v$. Following Lemma $6.6$ of \cite{hindman1998algebra}, $W \subseteq H$ and $V$ contains a copy of $\mathbb{H}$. Therefore we have a set $E \subseteq V$ of idempotents such that $|E| = 2^c$ and $u+v \neq u$ and $v+u \neq v$ for all distinct members $u$ and $v$ of $E$.

By Theorem $6.20$ of \cite{hindman1998algebra}, $(\beta \mathbb{N} +u) \bigcap  (\beta \mathbb{N} +v) = \phi$ whenever $u$ and $v$ are distinct members of $E$. So we can further say $(V+u)(V+v) = \phi$. For each $u \in E$ pick an idempotent $\alpha_u \in (p+V) \cap (V+u)$ with the property that $\alpha_u$ minimal in $V$.

By Corollary $2.6$ and Theorem $2.7$ of \cite{hindman1998algebra}, $p+V$ contains a minimal right ideal $R_1$ of $V$ and $V+u$ contains a minimal left ideal $L_1$ of $V$. Then, $R_1 \cap L_1$ is a group. Let $\alpha_u$ be the identity element of this group.) Then $\{\alpha_u : u \in E\}$ is a set of $2^c$ minimal ideampotents of $V$ in $p+V \subseteq R$.
\end{proof}
\begin{corollary} \label{cor:b}
Let $C$ be a C-set in $\mathbb{N}$. Then there exists a sequence $\langle C_n \rangle_{n=1}^{\infty}$ of pairwise disjoint C-sets in $\mathbb{N}$ such that $\bigcup_{n=1}^{\infty} C_i\subseteq C$.
\end{corollary}
\begin{proof}
By Theorem \ref{thA}, the set of idempotents in $\overline{C} \cap J(\mathbb{N})$ is infinite. Therefore $\overline{C} \cap J(\mathbb{N})$ contains an infinite strongly discrete subset. (Alternatively, there are two idempotents in $\overline{C} \cap J(\mathbb{N})$ so that $C$ can be divided into C-sets $C_1$ and $D_1$. Then $D_1$ can again be divided into two C-sets, $C_2$ and $D_2$ and so on.)
\end{proof}
\begin{corollary}\label{cor:c}
For each $n \in \mathbb{N}$, let $A_n$ be a strongly C-image partition regular matrix. Then the matrix $$M = \begin{pmatrix}
A_1 & 0 & 0 & \cdots \\
0 & A_2 & 0 & \cdots \\
0 & 0 & A_3 & \cdots \\
\vdots & \vdots & \vdots & \ddots \\
\end{pmatrix}$$ is also strongly C-image partition regular. 
\end{corollary}
\begin{proof}
Let $C$ be a C-set. By Corollary \ref{cor:b}, choose a sequence $\langle C_n \rangle_{n=1}^{\infty}$ of pairwise disjoint C-sets in $\mathbb{N}$ such that $\bigcup_{n=1}^{\infty}C_n \subseteq C$. For each $n \in \mathbb{N}$ choose $\overset{\rightarrow}{x}^{(n)} \in \mathbb{N}^{\omega}$ such that $\overset{\rightarrow}{y}^{(n)} = A_n \overset{\rightarrow}{x}^{(n)} \in C_n^{\omega}$ and if row $i$ and $j$ of $A_n$ fails to be equal, then $\overset{\rightarrow}{y}_i^{(n)} \neq \overset{\rightarrow}{y}_j^{(n)}$. Let $$\overset{\rightarrow}{z} = \begin{pmatrix}
\overset{\rightarrow}{x}^{(1)} \\
\overset{\rightarrow}{x}^{(2)} \\
\vdots \\
\end{pmatrix}.$$ Then all entries of $M \overset{\rightarrow}{z}$ are in $C$ and entries from distinct rows of $M \overset{\rightarrow}{z}$ are unequal.

Surely, Corollary \ref{cor:c} remains valid if \enquote{strongly C-image partition regular} is replaced by \enquote{C-image partition regular}. The same proof applies and one does not need to introduce the pairwise disjoint C-sets, which were required to guarantee that the entries of $M\overset{\rightarrow}{z}$ from distinct rows were distinct. 
\end{proof}
Notice that trivially, if $A$ is an $\omega \times \omega$ matrix with entries from $\mathbb{Q}$ and there is some positive rational number $m$ such that each row of $A$ sums to $m$, then $A$ is centrally image partition regular. (Given a  central set $C$, simply pick $d \in \mathbb{N}$ such that $d_m \in C$, which one can do because for each $n \in \mathbb{N}$, $N_n$ is a member of every idempotent by Lemma $6.6$ of \cite{hindman1998algebra}. Then let $x_i = d$ for each $i \in \omega$.)
\begin{theorem}
Let $k \in \mathbb{N}$ and $m \in \mathbb{Q}$ such that $m > 0$. Let $A$ be an $\omega \times \omega$ with rational entries   such that 
\begin{enumerate}
\item the sum of each row of $A$ is $m$, and 
\item for each $l \in \omega, \{\langle a_{i,o}, a_{i,1},\cdots, a_{i,l}\rangle : i \in \omega\}$ is finite.
\end{enumerate}
Let $\overset{\rightarrow}{r}^{(1)}, \overset{\rightarrow}{r}^{(2)}, \cdots , \overset{\rightarrow}{r}^{(k)} \in \mathbb{Q} ^{\omega} \setminus \overset{\rightarrow}{\{0\}}$ such that each $\overset{\rightarrow}{r}^{(i)}$ has only finitely many non-zero  entries. Then there exist $b_1, b_2, \cdots ,b_k \in \mathbb{Q}\setminus\{0\}$  such that 
$$\begin{pmatrix}
b_1 \overset{\rightarrow}{r}^{(1)} \\
b_2 \overset{\rightarrow}{r}^{(2)} \\
\vdots \\
b_k \overset{\rightarrow}{r}^{(k)} \\
A \\
\end{pmatrix}$$ is C-image partition regular.
\end{theorem}
\begin{proof}
Pick $l \in \mathbb{N}$ such that $r_i^{(j)} = 0$, for every $j \in \{1,2, \cdots k\}$ and every $i \geqslant l$. Let $s^{(j)} = \langle r_0^{(j)}, r_1^{(j)}, \cdots , r_l^{(j)} \rangle$, for $j \in \{1,2, \cdots k\}$. Enumerate $$\{\langle a_{i,o}, a_{i,1},\cdots, a_{i,l}\rangle : i \in \omega\}$$ as $\overset{\rightarrow}{w}^{(0)}, \overset{\rightarrow}{w}^{(1)}, \cdots , \overset{\rightarrow}{w}^{(u)}$. Let $d_i = m - \sum_{j=0}^{l-1} w_j^{(i)}$, for $i \in \{1,2, \cdots, u\}$. Let $E$ be the $(u+1)\times (l+1)$ matrix with entries 
\[
    e_{i,j}= 
\begin{cases}
    w_j^{(i)},& \text{if } j \in \{0,1, \cdots , l-1\}\\
    d_i,              & \text{if } j=l.
\end{cases}
\]
Then $E$ is image partition regular because $E$ has constant row sums. By applying Theorem $1.2(d)$ \cite{iprm2} $u+1$ times, pick $b_1, b_2, \cdots , b_k \in \mathbb{Q} \setminus \{0\}$ such that the matrix 
$$\begin{pmatrix}
b_1 \overset{\rightarrow}{s}^{(1)} \\
b_2 \overset{\rightarrow}{s}^{(2)} \\
\vdots \\
b_k \overset{\rightarrow}{s}^{(k)} \\
E \\
\end{pmatrix}$$
is image partition regulr. Hence by Theorem $1.2(b)$ of \cite{iprm2} the above matrix is image partition regular.

Let $C$ be a C-set and pick $\langle z_0, z_1, \cdots , z_l \rangle \in \mathbb{N}^{l+1}$ with $H\overset{\rightarrow}{z} \in C^{u+1}$. For $n \in \{0, 1, \cdots , l-1\}$, let $x_n =z_n$. For $n \in \{l, l+1, l+2, \cdots\}$, let $x_n = z_l$. Consequently, 
$$\begin{pmatrix}
b_1 \overset{\rightarrow}{r}^{(1)} \\
b_2 \overset{\rightarrow}{r}^{(2)} \\
\vdots \\
b_k \overset{\rightarrow}{r}^{(k)} \\
A \\
\end{pmatrix} \overset{\rightarrow}{x} \in C^{\omega}.$$
\end{proof}
\section{Some classes of C-image partition regular matrices} \label{sec:s:s}
We know that an extension of \enquote{first entries matrix} to infinite matrices does not essentially produce image partition regular matrices. Therefore we introduce a sparse version of the notion of first entries matrix which is studied quite elaborately in this section.

Firstly we recall the following definition which is definition $3.1$ in \cite{iprm2000hindman}.
\begin{definition}
Let A be an $\omega \times \omega$ matrix with rational entries. Then A is said to be a \textit{segmented
image partition regular matrix} if and only if
\begin{enumerate}
\item $A$ contains no row as $\overset{\rightarrow}{0}$;
\item the set $\{j \in \omega : a_{i,j} \neq 0\}$ is finite, for each $i \in \omega$; and
\item there is an increasing sequence $\langle \alpha_n \rangle_{n=0}^{\infty}$ with elements from $\omega$ satisfying  $\alpha_0 = 0$ and for each n $\in \omega$,
$$\{\langle a_{i,\alpha_n}, a_{i,\alpha_n+1}, a_{i,\alpha_n+2}, \cdots , a_{i,\alpha_{n+1}−1} \rangle : i \in \omega\} \setminus \{\overset{\rightarrow}{0}\}$$
is either empty or forms the set of rows of a finite image partition regular matrix.
\end{enumerate}
We shall say that A is a \textit{segmented first entries matrix} if each of these finite image partition regular matrices is a first entries matrix. Moreover A is said to be a \textit{monic segmented first entries matrix} if in addition  the first non-zero entry of each $\langle a_{i,\alpha_n}, a_{i,\alpha_{n+1}}, a_{i,\alpha_{n+2}}, \cdots , a_{i,\alpha_{n+1−1}} \rangle$, if any, is 1.
\end{definition} 
The most celebrated example of segmented first entries matrices are the finite sums matrix which generates the $(\mathcal{M},\mathcal{P},\mathcal{C})$-systems of \cite{prmpcs1993hindman}.
\begin{theorem}\label{thm3.2}
Any segmented image partition regular matrix is strongly C-image partition regular.
\end{theorem}
\begin{proof}
Let $\overset{\rightarrow}{c_0},\overset{\rightarrow}{c_1}, \overset{\rightarrow}{c_2}, \cdots$ denote the columns of a segmented image partition regular matrix $A$ and choose $\langle \alpha_n \rangle_{n=0}^{\infty}$ according to the definition of a segmented image partition regular matrix. Suppose $A_n$  is the matrix containing columns $\overset{\rightarrow}{c}_{\alpha_n}, \overset{\rightarrow}{c}_{{\alpha_n}+1}, \cdots , \overset{\rightarrow}{c}_{{\alpha_{n+1}-1}}$ for each $n \in \omega$. Then the set of non-zero rows of $A_n$ is finite and if it is non-empty then it is the set of rows of a finite image partition regular matrix. Let $B_n = (A_0~A_1~\cdots~A_n)$.

Take a C-set $C$ of $\mathbb{N}$. By Theorem \ref{thm1.8} choose an idempotent $p \in J(\mathbb{N})$ with $C \in p$. Let $C^* = {n \in C : −n + C \in p}$. Then $C^* \in p$ and $−n + C^* \in p$ for every $n \in C^*$ by [9, Lemma 4.14].

By Theorem \ref{thm1.10}, we can choose $\overset{\rightarrow}{x}^{(0)} \in \mathbb{N}^{\alpha_1 - \alpha_0}$ with the property that, if $\overset{\rightarrow}{y} = A_0 \overset{\rightarrow}{x}^{(0)}$, then $y_i \in C^*$ for every $i \in \omega$ with the $i$-th row of $A_0$ is non-zero, and entries of $\overset{\rightarrow}{y}$ which correspond to unequal rows of $A_0$ are distinct.

We now make the inductive assumption that, for some $m \in \omega$, we have chosen $\overset{\rightarrow}{x}^{(0)}, \overset{\rightarrow}{x}^{(1)}, \cdots , \overset{\rightarrow}{x}^{(m)}$ such that $\overset{\rightarrow}{x}^{(i)} \in \mathbb{N}^{\alpha_{i+1}-\alpha_i}$ for every $i \in \{0, 1, 2, \cdots ,m\}$, and, if $\overset{\rightarrow}{y} = B_m (\overset{\rightarrow}{x}^{(0)}, \overset{\rightarrow}{x}^{(1)}, \cdots , \overset{\rightarrow}{x}^{(m)})^t$
, then $y_j \in C^*$ for every $j \in \omega$ for which the $j$-th row of $B_m$ is non-zero and \enquote{t} denotes the matrix transpose. We further suppose that entries of $\overset{\rightarrow}{y}$ which correspond to unequal rows of $B_m$ are distinct.

Let $D = \{j \in \omega : \text{ row }j\text{ of }B_m+1\text{ is not }\overset{\rightarrow}{0}\}$. It follows that for each $j \in \omega$, $−y_j+C^* \in p$. (Either $y_j = 0$ or $y_j \in C^*$.) Let $l = \max\{y_i : i \in \omega\} + 1$ and note that $\mathbb{N}l \in p$ by [\cite{hindman1998algebra}, Lemma 6.6]. Thus by Theorem \ref{thm1.10}, we can choose $\overset{\rightarrow}{x}^{(m+1)} \in \mathbb{N}^{\alpha_{m+2}−\alpha_{m+1}}$ such that, if
$\overset{\rightarrow}{z} = A_{m+1}\overset{\rightarrow}{x}^{(m+1)}$, then $z_j \in \mathbb{N}l \setminus \cap \bigcap_{t \in D} (−y_t+C^*)$ for every $j \in D$, and $z_j \neq z_k$ whenever rows $j$ and $k$ of $A_{m+1}$ are distinct and not equal to $\overset{\rightarrow}{0}$. Since each $z_j \in \mathbb{N}l$, we also get that $y_j + z_j \neq y_k + z_k$ whenever $j$, $k$ $\in D$ and distinct rows $j$ and $k$ of $B_{m+1}$.

Thus we can choose an infinite sequence $\langle \overset{\rightarrow}{x}^{(i)}\rangle_{i \in \omega}$ with the property that for every $i \in \omega$, $\overset{\rightarrow}{x}^{(i)} \in \mathbb{N}^{\alpha_{i+1}-\alpha_i}$, and, if $\overset{\rightarrow}{y} = B_i (\overset{\rightarrow}{x}^{(0)}, \overset{\rightarrow}{x}^{(1)}, \cdots , \overset{\rightarrow}{x}^{(i)})^t$, then $y_j \in C^*$ for every $j \in \omega$ for which the $j$-th row of $B_i$ is non-zero. Moreover, entries of $\overset{\rightarrow}{y}$ corresponding to distinct rows of $B_i$ are distinct.

Let $\overset{\rightarrow}{y} = A\overset{\rightarrow}{x}$ where $\overset{\rightarrow}{x} = (\overset{\rightarrow}{x}^{(0)}, \overset{\rightarrow}{x}^{(1)}, \overset{\rightarrow}{x}^{(3)}, \cdots )^t$. We note that, for every $j \in \omega$ and for $i > m$, there exists $m \in \omega$ such that $y_j$ is the $j$-th entry of $B_i(\overset{\rightarrow}{x}^{(0)}, \overset{\rightarrow}{x}^{(1)}, \overset{\rightarrow}{x}^{(3)}, \cdots , \overset{\rightarrow}{x}^{(i)})^t$. Thus all the entries of $\overset{\rightarrow}{y}$ are in $C^*$ and entries corresponding to distinct rows are distinct.
\end{proof}

Now we recall the following definition which is Definition $4.1$ in \cite{iprm2000hindman}.
\begin{definition} An $\omega \times \omega$ matrix $A$ is said to be a \textit{restricted triangular matrix} if
and only if all entries of $A$ are from $\mathbb{Z}$ and there exist $d \in \mathbb{N}$ and an increasing function $j : \omega \to \omega$ such that for all $i \in \omega$,
\begin{enumerate}
\item $a_{i, j(i)} \in \{1, 2, \cdots , d\}$,
\item $a_{i,l} = 0$, whenever $l > j(i)$, and
\item for all $k > i$ and all $t \in \{1, 2, \cdots , d\}$, $t|_{a_{k,j(i)}}$.
\end{enumerate}
\end{definition}
\begin{theorem}
A restricted triangular matrix $A$ is strongly C-image partition regular. In particular, if $p \in \bigcap_{ n \in \mathbb{N}} cl_{\beta \mathbb{N}}(n\mathbb{N})$ and $P \in p$, then there exists $\overset{\rightarrow}{x} \in \mathbb{N}^{\omega}$ such that the entries of $A \overset{\rightarrow}{x}$ are distinct elements of $P$.
\end{theorem}
\begin{proof}
The proof is essentially done in the proof of Theorem $4.2$ in \cite{iprm2000hindman}. 
\end{proof}
\begin{corollary}
Let $A$ be an $\omega \times \omega$ matrix with entries from $\mathbb{Z}$. Suppose exists an increasing function $j : \omega \to \omega$ such that for all $i \in \omega$ the following properties are satisfied:
\begin{enumerate}
\item $a_{i,j(i)} = 1$ and
\item $a_i,l = 0$ for all $l > j(i)$.
\end{enumerate}
Then $A$ is strongly C-image partition regular.
\end{corollary}
\begin{proof}
The corollary is immediate because $A$ is a restricted triangular matrix with $d = 1$.
\end{proof}
\begin{corollary}
Let $A$ be an $\omega \times \omega$ matrix with entries from $\mathbb{Z}$. In addition suppose $A$ contains only finitely many non-zero entries in each row. Suppose there exist $d \in \mathbb{N}$ and a function $j : \omega \to \omega$ such that for all $i \in \omega$, the following are satisfied:
\begin{enumerate}
\item $a_{i,j(i)} \in \{1, 2, \cdots , d\}$ and
\item $a_{k,j(i)} = 0$ for all $k \neq i$.
\end{enumerate}
Then A is strongly C-image partition regular.
\end{corollary}
\begin{proof}
The proof of this corollary follows by noting a possible rearrangement of the columns using condition ($2$).
\end{proof}
\begin{theorem}
Let $A$ be a restricted triangular matrix with finitely many non-zero entries. Let $\overset{\rightarrow}{r} \in \mathbb{Z}^{\omega} \setminus \{\overset{\rightarrow}{0}\}$. Then $\begin{pmatrix}b\overset{\rightarrow}{r} \\ A \end{pmatrix}$ is a strongly C-image partition regular matrix for some $b \in \mathbb{Q} \setminus \{0\}$.
\end{theorem}
\begin{proof}
Pick $d \in \mathbb{N}$ and $j : \omega \to \omega$ according as Definition $1.3$. Take $l \geqslant j(0)$ such
that $r_i = 0$ for all $i > l$. Also pick $\gamma \in \omega$ such that $j(\gamma) \leqslant l < j(\gamma + 1)$.

Call $B$ to be the upper left $(\gamma +1)\times(l+1)$ corner of $A$. By Theorem $1.4$, $A$ is C-image partition regular and therefore $B$ is image partition regular. Applying Theorem \ref{thm3.2}, $l+2$ times, pick $b_0, b_1, \cdots , b_l, b$ in $\mathbb{Q}$ such that 
$$D=\begin{pmatrix}br_0 & br_1 & br_2 & \cdots & br_l\\ b_0 & 0 & o & \cdots & 0 \\ 0 & b_1 & o & \cdots & 0 \\ 0 & 0 & b_2 & \cdots & 0 \\ \vdots & \vdots & \vdots & \ddots & \vdots \\ 0 & 0 & 0 & \cdots & b_l \\ & & B & & \end{pmatrix}$$ is image partition regular. We claim that $\begin{pmatrix}b\overset{\rightarrow}{r} \\ A \end{pmatrix}$ is C-image partition regular. For that let $C$ be a C-set and let $c$ be a common multiple of the numerators of $b_0, b_1, \cdots , b_l$.
Then $C \cap \mathbb{N}cd!$ is again a C-set. By Theorem \ref{thm3.2}, pick $x_0, x_1, \cdots , x_l$ such that all entries of $D \begin{pmatrix}x_0\\ x_1 \\ \vdots \\ x_l \end{pmatrix}$ are in $C \cap \mathbb{N}cd!$ and are distinct. For $t \in \{0, 1, \cdots , l\}$, one has in particular that $b_tx_t \in \mathbb{N}cd!$ and consequently $x_t \in \mathbb{N}d!$. For $t > l$, choose $x_t= d!$ exactly as in the proof of Theorem $1.4$. One concludes immediately that all entries of $\begin{pmatrix}b\overset{\rightarrow}{r} \\ A \end{pmatrix} \overset{\rightarrow}{x}$ are in $C$ and are unequal.
\end{proof}
\begin{theorem}
Let $A$ be a C-image partition regular matrix and let $\langle b_n \rangle_{n=0}^{\infty}$ be a sequence of positive integers. Let $$B = \begin{pmatrix} b_0 & 0 & o & \cdots \\ 0 & b_1 & o & \cdots \\ 0 & 0 & b_2 & \cdots \\ \vdots & \vdots & \vdots & \ddots \end{pmatrix}. \text{ Then the matrix } \begin{pmatrix}
\textbf{O} & B \\ A & \textbf{O} \\ A & B \end{pmatrix}$$ is C-image partition regular.
\end{theorem}
\begin{proof}
Let $C$ be a C-set of $\mathbb{N}$. Take a minimal idempotent $p$ in $\beta \mathbb{N}$ such that $C \in p$. Let $D := \{x \in C : −x + C \in p\}$. Then by [\cite{hindman1998algebra}, Lemma $4.14$] $D \in p$ and therefore $D$ is C-set. So we can get $\overset{\rightarrow}{x} \in N^{\omega}$ such that $A \overset{\rightarrow}{x} \in D^{\omega}$.

Define $c_n = \sum_{t=0}^{\infty}a_{n,t} \cdot  x_t$, for any given $n \in \omega$. Then $C \cap (−c_n + C) \in p$, so pick $z_n \in
C \cap (−c_n + C) \cap \mathbb{N}b_n$ and let $y_n = \frac{z_n}{b_n}$. Thus we get $$\begin{pmatrix}
\textbf{O} & B \\ A & \textbf{O} \\ A & B \end{pmatrix} \begin{pmatrix} \overset{\rightarrow}{x} \\ \overset{\rightarrow}{y} \end{pmatrix} \in C^{\omega + \omega + \omega}.$$
\end{proof}
Let us quickly recall the following definition which is Definition $4.8$ in \cite{iprm2000hindman}.
\begin{definition}
Let $C$ be a  $\gamma \times \delta$ matrix with finitely many non-zero entries in each row, for some $ \gamma, \delta \in \omega \cup \{\omega\}$. For each $t < \delta$, let $B_t$ be a $u_t \times v_t$ (finite) matrix. Let $R =
\{(i, j) : i < \gamma \text{ and } j \in \bigtimes_{ t < \delta}\{0, 1, \cdots , u_t − 1\}\}$. Given $t < \delta$ and
$k \in \{0, 1, \cdots , u_t − 1\}$, denote the $k$-th row of $B_t$ by the notation $\overset{\rightarrow}{b}_{k}^{(t)}$ . Then $D$ is said to be an insertion matrix of $\langle B_t \rangle_{t<\delta}$ into $C$ if and only if the rows of $D$ are all rows of the form $$c_{i,0} \cdot \overset{\rightarrow}{b}_{j(0)}^{(0)} \frown c_{i,1} \cdot \overset{\rightarrow}{b}_{j(1)}^{(1)} \frown \cdots$$ where $(i,j) \in R$.
\end{definition}
For example we can consider that one which is given in \cite{iprm2000hindman}. Suppose $C =
\begin{pmatrix}
1 & 0\\
2 & 1
\end{pmatrix}$
, $B_0 =
\begin{pmatrix}
1 & 1 \\
5 & 7
\end{pmatrix}$
, and $B_1 =
\begin{pmatrix}
0 & 1 \\ 
3 & 3 
\end{pmatrix}$
, then the following matrix
$$D = \begin{pmatrix}
1 & 1 & 0 & 0 \\
5 & 7 & 0 & 0 \\
2 & 2 & 0 & 1 \\
2 & 2 & 3 & 3 \\
10 & 14 & 0 & 1 \\
10 & 14 & 3 & 3 \\
\end{pmatrix}$$
is an insertion matrix of $\langle B_t \rangle_{t < 2}$ into $C$.
\begin{theorem}
Let $C$ be a segmented first entries matrix. Also let $B_t$ be a $u_t \times v_t$ (finite) image partition regular matrix, for each $t < \omega$. Then any insertion matrix of $\langle B \rangle_{t<\omega}$ into $C$ is C-image partition regular.
\end{theorem}
\begin{proof}
Let us take $A$ to be an insertion matrix of $\langle B \rangle_{t<\omega}$ into $C$. For each $t \in \omega$, pick by
Theorem $1.5(g)$ of \cite{iprm2000hindman}, some $m_t \in \mathbb{N}$ and a $u_t \times m_t$ first entries matrix $D_t$ with the property that for all $\overset{\rightarrow}{y} \in \mathbb{N}^{m_t}$ there exists $ \overset{\rightarrow}{x} \mathbb{N}^{v_t}$ such that $B_t \overset{\rightarrow}{x} = D_t \overset{\rightarrow}{y}$. Let $E$ be another insertion matrix of $\langle D_t \rangle_{t<\omega}$ into $C$ where the rows occur in the corresponding position to those of $A$. That is, if $i < \omega$ and $j \in \bigtimes_{t<\omega} \{0, 1, \cdots , u_t − 1\}$ and $$c_{i,0} \cdot \overset{\rightarrow}{b}_{j(0)}^{(0)} \frown c_{i,1} \cdot \overset{\rightarrow}{b}_{j(1)}^{(1)} \frown \cdots$$ is row $k$ of $A$, then $$c_{i,0} \cdot \overset{\rightarrow}{d}_{j(0)}^{(0)} \frown c_{i,1} \cdot \overset{\rightarrow}{d}_{j(1)}^{(1)} \frown \cdots$$ is row $k$ of $E$.

Let $H$ be a C-set of $\mathbb{N}$. By Lemma $4.9$ of \cite{iprm2000hindman}, $E$ is a segmented first entries
matrix. Therefore, pick $\overset{\rightarrow}{y} \in \mathbb{N}^{\omega}$ such that all entries of $E\overset{\rightarrow}{y}$ are contained in $H$. Let $\delta_0 =  \gamma_0 = 0$ and for $n \in \mathbb{N}$ take $\delta_n := \sum_{t=0}^{n−1} v_t$ and $\gamma_n := \sum_{t=0}^{n−1} m_t$. For each $n \in \omega$, pick
$$\begin{pmatrix}
x_{\delta_n} \\ x_{\delta_{n}+1}\\ \vdots \\ x_{\delta_{n+1}-1}
\end{pmatrix} \in \mathbb{N}^{v_n} \text{ such that } B_t \begin{pmatrix}
x_{\delta_n} \\ x_{\delta_{n}+1}\\ \vdots \\ x_{\delta_{n+1}-1}
\end{pmatrix} = D_t \begin{pmatrix}
y_{\gamma_n} \\ y_{\gamma_{n}+1}\\ \vdots \\ y_{\gamma_{n+1}-1}
\end{pmatrix}.$$ Then it is clear that $A \overset{\rightarrow}{x} = E \overset{\rightarrow}{y}$.
\end{proof}
At the end of the paper we raise the following question.

\textbf{Question:} Is it true that every C-image partition regular matrices are Centrally image partition regular?

\bibliographystyle{abbrvnat}
\bibliography{bibfile}

\begin{thebibliography}{9}
\providecommand{\natexlab}[1]{#1}
\providecommand{\url}[1]{\texttt{#1}}
\expandafter\ifx\csname urlstyle\endcsname\relax
  \providecommand{\doi}[1]{doi: #1}\else
  \providecommand{\doi}{doi: \begingroup \urlstyle{rm}\Url}\fi

\bibitem[De et~al.(2008)De, Hindman, and Strauss]{de2008new}
D.~De, N.~Hindman, and D.~Strauss.
\newblock A new and stronger central sets theorem.
\newblock \emph{Fund. Math}, 199\penalty0 (2):\penalty0 155--175, 2008.

\bibitem[Deuber et~al.(1995)Deuber, Hindman, Leader, and Lefmann]{deu95inf}
W.~A. Deuber, N.~Hindman, I.~Leader, and H.~Lefmann.
\newblock Infinite partition regular matrices.
\newblock \emph{Combinatorica}, 15\penalty0 (3):\penalty0 333--355, 1995.
\newblock ISSN 0209-9683.
\newblock URL \url{https://doi.org/10.1007/BF01299740}.

\bibitem[Furstenberg(1981)]{frus81}
H.~Furstenberg.
\newblock \emph{Recurrence in ergodic theory and combinatorial number theory}.
\newblock Princeton University Press, Princeton, N.J., 1981.
\newblock ISBN 0-691-08269-3.
\newblock M. B. Porter Lectures.

\bibitem[Hindman and Lefmann(1993)]{prmpcs1993hindman}
N.~Hindman and H.~Lefmann.
\newblock Partition regularity of {$(M,P,C)$}-systems.
\newblock \emph{J. Combin. Theory Ser. A}, 64\penalty0 (1):\penalty0 1--9,
  1993.
\newblock ISSN 0097-3165.
\newblock URL \url{https://doi.org/10.1016/0097-3165(93)90084-L}.

\bibitem[Hindman and Strauss(1998)]{hindman1998algebra}
N.~Hindman and D.~Strauss.
\newblock \emph{Algebra in the Stone-{\v{C}}ech compactification: theory and
  applications}, volume~27.
\newblock Walter de Gruyter, 1998.

\bibitem[Hindman and Strauss(2000)]{iprm2000hindman}
N.~Hindman and D.~Strauss.
\newblock Infinite partition regular matrices. {II}. {E}xtending the finite
  results.
\newblock In \emph{Proceedings of the 15th {S}ummer {C}onference on {G}eneral
  {T}opology and its {A}pplications/1st {T}urkish {I}nternational {C}onference
  on {T}opology and its {A}pplications ({O}xford, {OH}/{I}stanbul, 2000)},
  volume~25, pages 217--255 (2002), 2000.

\bibitem[Hindman et~al.(2003)Hindman, Leader, and Strauss]{iprm2}
N.~Hindman, I.~Leader, and D.~Strauss.
\newblock Infinite partition regular matrices: solutions in central sets.
\newblock \emph{Trans. Amer. Math. Soc.}, 355\penalty0 (3):\penalty0
  1213--1235, 2003.
\newblock ISSN 0002-9947.
\newblock URL \url{https://doi.org/10.1090/S0002-9947-02-03191-4}.

\bibitem[Schur(1917)]{schur1917kongruenz}
I.~Schur.
\newblock {\"U}ber kongruenz x...(mod. p.).
\newblock \emph{Jahresbericht der Deutschen Mathematiker-Vereinigung},
  25:\penalty0 114--116, 1917.

\bibitem[Van~der Waerden(1927)]{van1927beweis}
B.~L. Van~der Waerden.
\newblock Beweis einer baudetschen vermutung.
\newblock \emph{Nieuw Arch. Wiskunde}, 15:\penalty0 212--216, 1927.

\end{thebibliography}
\end{document}